\documentclass[reqno]{amsart}
\usepackage{amsfonts}
\usepackage{amssymb}
\usepackage{times}
\usepackage{xcolor}
\def\g2{\hbox{\mathfrak g}_2}
\def\f4{\hbox{\mathfrak f}_4}
\def\e6{\hbox{\mathfrak e}_6}
\def\fs4{\hbox{\mathfrak fs}_4}
\def\F4{\hbox{\mathfrak F}_4}
\def\Fs4{\hbox{\mathfrak Fs}_4}

\def\al{\ifcase\xypolynode\or F \or A\or B\or C\or D\or G\fi}
\def\ala{\ifcase\xypolynode\or a \or b\or c\or d\or g\or f\fi}

\makeatletter
\def\Ddots{\mathinner{\mkern1mu\raise\p@
\vbox{\kern7\p@\hbox{.}}\mkern2mu
\raise4\p@\hbox{.}\mkern2mu\raise7\p@\hbox{.}\mkern1mu}}
\makeatother

{\theoremstyle{plain}%
  \newtheorem{theorem}{Theorem}[section]
  \newtheorem{corollary}{Corollary}[section]
  \newtheorem{proposition}{Proposition}[section]
  \newtheorem{lemma}{Lemma}[section]
  \newtheorem{example}{Example}[section]
  
  \newtheorem{definition}{Definition}[section]

\newtheorem{remark}{Remark}[section]

\newfont{\hueca}{msbm10}
\def\hu #1{\hbox{\hueca #1}}\def\hu #1{\hbox{\hueca #1}}
\begin{document}

\title{On split regular Hom-Lie superalgebras}

\thanks{The first, second and fourth authors acknowledge financial assistance by the Centre for Mathematics of the University of Coimbra -- UID/MAT/00324/2013, funded by the Portuguese Government through FCT/MEC and co-funded by the European Regional Development Fund through the Partnership Agreement PT2020. Third and fourth authors are supported by the PCI of the UCA `Teor\'\i a de Lie y Teor\'\i a de Espacios de Banach', by the PAI with project numbers FQM298, FQM7156 and by the project of the Spanish Ministerio de Educaci\'on y Ciencia  MTM2013-41208P. The fourth author acknowledges the Funda\c{c}\~{a}o para a Ci\^{e}ncia e a Tecnologia for the grant with reference SFRH/BPD/101675/2014.}

\author[H. Albuquerque ]{Helena~Albuquerque}

\address{Helena~Albuquerque, CMUC, Departamento de Matem\'atica, Universidade de Coimbra, Apartado 3008,
3001-454 Coimbra, Portugal. \hspace{0.1cm} {\em E-mail address}: {\tt lena@mat.uc.pt}}

\author[E. Barreiro]{Elisabete~Barreiro}

\address{Elisabete~Barreiro, CMUC, Departamento de Matem\'atica, Universidade de Coimbra, Apartado 3008,
3001-454 Coimbra, Portugal. \hspace{0.1cm} {\em E-mail address}: {\tt mefb@mat.uc.pt}}{}

\author[A.J. Calder\'on]{A.J.~Calder\'on}

\address{A.J.~Calder\'on, Departamento de Matem\'aticas, Universidad de C\'adiz, Campus de Puerto Real, 11510, Puerto Real, C\'adiz, Espa\~na. \hspace{0.1cm} {\em E-mail address}: {\tt ajesus.calderon@uca.es}}{}

\author[Jos\'{e} M. S\'{a}nchez]{Jos\'{e} M. S\'{a}nchez}

\address{Jos\'{e} M. S\'{a}nchez, CMUC, Departamento de Matem\'atica, Universidade de Coimbra, Apartado 3008, 3001-454 Coimbra, Portugal. \hspace{0.1cm} {\em E-mail address}: {\tt txema.sanchez@mat.uc.pt}}

\begin{abstract}
We introduce the class of split regular Hom-Lie superalgebras as the natural extension of the one of split Hom-Lie algebras and Lie superalgebras, and study its structure by showing that an arbitrary split regular Hom-Lie superalgebra ${\frak L}$ is of the form ${\frak L} = U + \sum_j I_j$ with $U$ a linear subspace of a maximal abelian graded subalgebra $H$ and any $I_j$ a well described (split) ideal of ${\frak L}$ satisfying $[I_j,I_k] = 0$ if $j \neq k$. Under certain conditions, the simplicity of ${\frak L}$ is characterized and it is shown that ${\frak L}$ is the direct sum of the family of its simple ideals.

{\it Keywords}: Hom-Lie superalgebra, Root, Root space, Structure theory.

{\it 2000 MSC}: 17A60, 17A70, 17B22.

\end{abstract}

\maketitle

%%%%%%%%%%%%%%%%%%%%%%%%%%%%%%%%%%%%%%%%%%%%%%%%%%%%%%%%%%%%%%
%%%%%%%%%%%%%%%%%%%%%%%%%%%%%%%%%%%%%%%%%%%%%%%%%%%%%%%%%%%%%%
\section{Introduction and first definitions}
%%%%%%%%%%%%%%%%%%%%%%%%%%%%%%%%%%%%%%%%%%%%%%%%%%%%%%%%%%%%%%
%%%%%%%%%%%%%%%%%%%%%%%%%%%%%%%%%%%%%%%%%%%%%%%%%%%%%%%%%%%%%%

The motivation to study Hom-Lie structures are related to physics and to deformations of Lie algebras (see for instance \cite{Maksuper, Arnlind, Asso2, Tong, Makaso, Sheng, Yau}). A Hom-Lie superalgebra is a $\mathbb{Z}_2$-graded space with a bracket for which the super Jacobi identity is twisted by a homomorphism. This class was introduced by Ammar and Makhlouf in \cite{Maksuper}, where they have shown that the supercommutator bracket defined using the multiplication in a Hom-associative superalgebra leads naturally to a Hom-Lie superalgebra. Hom-Lie superalgebras are gene\-ralizations of Lie superalgebras, Lie algebras as well as Hom-Lie algebras, and also they are a particular case of $\Gamma$-graded quasi-Lie algebras introduced by Larsson and Silvestrov in \cite{Origin}, who also have shown its relation with discrete and deformed vector fields and diffe\-rential calculus. We also note that an analogous notion has been considered in other categories like Hom-associative algebras, Hom-alternative algebras,  Hom-Leibniz algebras, etc. (see \cite{Maksuper, Cheng, Amin28, Mak10, Amin34, Mak20, Donald}).

In the present paper we introduce the class of split regular   Hom-Lie superalgebras as the natural extension of the one of split regular Hom-Lie algebras (see \cite{HomLie}) and study its structure presenting them as a direct sum of adequate ideals. We also characterize a simple split regular split Hom-Lie superalgebra which together a relation of equivalence defined in its root system allow us to obtain a final expression as direct sum of simple ideals. Here it is interesting  to mention the recent references \cite{Cao1, Cao2, Cao3, Cao4} where different classes of split  Hom-algebras are also studied  from the viewpoint of their inner structures. 

We briefly describe the content of this paper. In Section 2 we develop connection of roots techniques which becomes the main tool in our study. In Section 3 we apply all of the machinary introduced in the previous section to show that a split regular Hom-Lie superalgebra ${\frak L}$ is of the form ${\frak L} = U + \sum_j I_j$ with $U$ a linear subspace of a maximal abelian graded subalgebra $H$ and any $I_j$ a well described ideal of ${\frak L}$ satisfying $[I_j, I_k]=0$ if $j\neq k$. In the final section, under certain conditions, the simplicity of ${\frak L}$ is characterized and it is shown that ${\frak L}$ is the direct sum of the family of its simple ideals.

\begin{definition}\rm
A {\it Hom-Lie superalgebra} ${\frak L}$ is a ${\hu Z}_2$-graded algebra ${\frak L} = {\frak L}_{\bar 0} \oplus {\frak L}_{\bar 1}$ over a base field ${\hu K}$ endowed with a bilinear product $$[\cdot, \cdot] : {\frak L} \times {\frak L} \to {\frak L}$$ and an even superspace homomorphism $\phi : {\frak L} \to {\frak L},$ satisfying
\begin{equation}
[x,y]=-(-1)^{{\bar i} {\bar j}}[y,x]
\end{equation}
\begin{equation}
(-1)^{{\bar i} {\bar k}}[[x,y], \phi(z)] + (-1)^{{\bar i} {\bar j}}[[y,z],\phi(x)] + (-1)^{{\bar j} {\bar k}}[[z,x],\phi(y)] = 0,
\end{equation}
for all homogeneous elements $x \in {\frak L}_{{\bar i}}, y \in {\frak L}_{\bar j}, z \in {\frak L}_{\bar k},$ with ${\bar i}, {\bar j}, {\bar k} \in {\hu Z}_2,$ called {\it skew-supersymmetry} and {\it super Hom-Jacobi identity}, respectively. Furthermore, if $\phi$ is an superalgebra automorphism ${\frak L}$ is called a {\it regular Hom-Lie superalgebra}.
\end{definition}

\noindent Note that ${\frak L}_{\bar 0}$ is a Hom-Lie algebra called the even or bosonic part of $\frak L$ while ${\frak L}_{\bar 1}$ is called the odd or fermonic part of $\frak L$. The usual  regularity concepts will be understood in the graded sense. A {\it graded subalgebra} $A$ of ${\frak L}$ is a graded subspace $A = A_{\bar 0} \oplus A_{\bar 1}$ of ${\frak L}$ such that $[A,A] \subset A$ and $\phi(A) = A$. An {\it ideal} $I$ of ${\frak L}$ is a graded subalgebra $I=I_{\bar 0} \oplus I_{\bar 1}$ of ${\frak L}$ satisfying $[I,{\frak L}] \subset I.$ A Hom-Lie superalgebra $\frak L$ is called {\it simple} if $[{\frak L},{\frak L}] \neq 0$ and its only (graded) ideals are $\{0\}$ and ${\frak L}$. Throughout this paper we consider regular Hom-Lie superalgebras and denote by $\mathbb{N}_0$ the set of all non-negative integers and by $\mathbb{Z}$ the set of all integers. Finally, we emphasize that the regular Hom-Lie superalgebras ${\frak L}$ are considered of arbitrary dimension and over an arbitrary base field ${\hu K}$.
Let us introduce the class of split algebras in the framework of regular Hom-Lie superalgebras in a similar way to the cases of Hom-Lie algebras, Lie algebras and Lie superalgebras, among other classes of algebras. We begin by considering a maximal abelian graded subalgebra $H = H_{\bar{0}} \oplus H_{\bar{1}}$ among the abelian graded subalgebras of $\frak L.$

\begin{definition}\label{split}\rm
Let $\frak L$ be a regular Hom-Lie superalgebra and denote by $H=H_{\bar 0}\oplus H_{\bar 1}$ a maximal abelian graded subalgebra (MAGSA) of $\frak L$. For a linear functional $\alpha : H_{\bar 0} \to \hu{K},$ we define the {\it root space} of $\frak L$ (with respect to $H$) associated to $\alpha$ as the subspace
$${\frak L}_{\alpha} := \bigl\{v_{\alpha} \in {\frak L} : [h_{\bar 0},v_{\alpha}] = \alpha(h_{\bar 0})\phi(v_{\alpha}) \hspace{0.1cm} {\it for} \hspace{0.1cm} {\it any} \hspace{0.1cm} h_{\bar 0} \in H_{\bar 0}\bigr\}.$$ The elements $\alpha \in (H_{\bar 0})^*$ satisfying ${\frak L}_{\alpha} \neq 0$ are called {\it roots} of ${\frak L}$ (with respect to $H$) and we denote $\Lambda := \{\alpha \in (H_{\bar 0})^* \setminus \{0\}: {\frak L}_{\alpha}\neq 0\}$ the {\it root system} of ${\frak L}.$ We say that ${\frak L}$ is a {\it split regular Hom-Lie superalgebra} (with respect to $H$) if $${\frak L} = H \oplus \Bigl(\bigoplus\limits_{\alpha \in \Lambda}{\frak L}_{\alpha}\Bigr).$$
\end{definition}

\noindent Split regular Hom-Lie algebras, split Lie superalgebras and split Lie algebras are examples of split regular Hom-Lie superalgebras. Hence, the present paper extends the results in \cite{HomLie, YoLi, Nosalg}. Let us now show a proper example.

\begin{example}\label{example1}
Consider the real  ${\mathbb Z}_2$-graded vector space
${\frak L}={\frak L}_{\bar 0} \oplus {\frak L}_{\bar 1}$,  with
basis $$\{e_1,e_2\} \cup \{h_n,x_n,y_n: n \in {\mathbb N}, n \geq 2\} $$ of ${\frak L}_{\bar 0}$ and  $$\{e_3\}\cup \{f_n,g_n: n \in {\mathbb N}, n \geq 2\} $$
of ${\frak L}_{\bar 1};$ and   where the nonzero products on these elements
are induced  by the relations:
$$[e_2,e_1]=e_1, \hspace{0.3cm} [h_n,x_n]=2n^2x_n, \hspace{0.3cm}
[h_n,y_n]=- \frac{2}{n^2}y_n,$$ $$[x_n,y_n]=h_n,\hspace{0.3cm}
[y_n,g_n]=\frac{1}{n}f_n, \hspace{0.3cm}[x_n,f_n]={n}g_n,$$
$$[h_n,f_n]=-\frac{1}{n}f_n,\hspace{0.3cm} [h_n,g_n]={n}g_n,
 \hspace{0.3cm}[g_n,f_n]=h_n,\hspace{0.3cm}  [g_n,g_n]=-2n^2x_n,$$
 $$[f_n,f_n]=\frac{2}{n^2}y_n.$$

  Then by considering the  superspace homomorphism $$\phi: {\frak L} \to {\frak L}$$ defined by

  $$\hbox{$\phi(e_i)=e_i$,  $i \in \{1,2,3\}$ and }$$
  $$\hbox{$\phi(h_n)=h_n$, $\phi(x_n)=n^2 x_n$, $\phi(y_n)=\frac{1}{n^2} y_n$, $\phi(f_n)=\frac{1}{n} f_n$, $\phi(g_n)=n g_n$ }$$
   for any $n \geq 2$,
  we have that
  ${\frak L}={\frak L}_{\bar 0} \oplus {\frak L}_{\bar
 1}$ becomes a, (non-Hom-Lie algebra  and  non-Lie superalgebra), split regular  Hom-Lie superalgebra

 \begin{equation}\label{sep1}
 {\frak L}=H \oplus {\frak L}_{\alpha} \oplus (\bigoplus\limits_{n\in {\mathbb N}, n \geq 2}  {\frak
 L}_{\beta_n})\oplus (\bigoplus\limits_{n\in {\mathbb N}, n \geq 2}  {\frak
 L}_{-\beta_n}) \oplus (\bigoplus\limits_{n\in {\mathbb N}, n \geq 2}  {\frak
 L}_{\gamma_n})\oplus (\bigoplus\limits_{n\in {\mathbb N}, n \geq 2}  {\frak
 L}_{-\gamma_n})
 \end{equation}
 where
 % $H={\rm span}_{\mathbb R} \{ e_1,e_2\} \cup \{h_n: n\in {\mathbb N}, n \geq 2\}, $

 $$H=\langle e_2,h_2,h_3,...,h_n,... \rangle \oplus \langle e_3 \rangle,$$
 $$\hbox{ ${\frak L}_{\alpha}=\langle e_1 \rangle,$
  ${\frak L}_{\beta_n}=\langle x_n
 \rangle$, ${\frak L}_{-\beta_n}=\langle y_n
 \rangle$, ${\frak L}_{\gamma_n}=\langle f_n
 \rangle$ and ${\frak L}_{-\gamma_n}=\langle g_n
 \rangle$},$$ being
 $$\alpha, \beta_n, \gamma_n: H \to {\mathbb R},$$
$ n\in {\mathbb N}, n \geq 2,$
  defined by

    $$\hbox{$\alpha(e_2)= 1,$ and  $\alpha(e_3)= \alpha(h_n)=0,$ for any $n \geq 2$},$$
     $$\hbox{$\beta_n(e_2)= \beta_n(e_3)=0,$  $ \beta_n(h_n)=2,$ and $ \beta_n(h_m)=0$ when $n\neq m$},$$
     %$$\hbox{$\gamma_n(e_2)= \gamma_n(e_3)=0,$  $ \gamma_n(h_n)=-2n^2,$ and $ \gamma_n(h_m)=0$ when $n\neq m$},$$
     $$\hbox{$\gamma_n(e_2)= \gamma_n(e_3)=0,$  $ \gamma_n(h_n)=-1,$ and $ \gamma_n(h_m)=0$ when $n\neq m$}.$$
    % $$\hbox{$\tau_n(e_2)= \tau_n(e_3)=0,$  $ \tau_n(h_n)=-1,$ and $ \tau_n(h_m)=0$ when $n\neq m$}.$$

\end{example}

\bigskip

It is clear that the root space associated to the zero root satisfies $H \subset {\frak L}_0.$ Conversely, given any $v_0\in {\frak L}_0$ we can write $v_0 = h + \sum_{i=1}^n v_{\alpha_i}$ with $h \in H$ and $v_{\alpha_i} \in {\frak L}_{\alpha_i}$ for $i \in \{1,\dots,n\}$, being $\alpha_i \in \Lambda$ with $\alpha_i \neq \alpha_j$ if $i\neq j$. Hence $0 = [h_{\bar 0}, h + \sum_{i=1}^n v_{\alpha_i}] = \sum_{i=1}^n \alpha_i(h_{\bar 0})\phi(v_{\alpha_i})$ for any $h_{\bar 0} \in H_{\bar 0}$. So, taking into account the direct character of the sum, that $\alpha_i \neq 0$ and $\phi$ is an automorphism, we have that any $v_{\alpha_i} = 0$ and then $v_0 \in H$. Consequently $$H = {\frak L}_0.$$

By the grading of ${\frak L},$ any $v_{\alpha} \in {\frak L}_{\alpha}, \alpha \in \Lambda \cup \{0\}$ may be expressed in the form $v_{\alpha} = v_{\alpha,{\bar 0}}+v_{{\alpha},{\bar 1}}$ with $v_{\alpha, {\bar i}} \in {\frak L}_{\bar i}, {\bar i} \in {\hu Z}_2$, then
$$[h_{\bar 0},v_{\alpha, {\bar i}}] = \alpha(h_{\bar 0})\phi(v_{\alpha, {\bar i}})$$ for any $h_{\bar 0} \in H_{\bar 0}$. From here, ${\frak L}_{\alpha} = ({\frak L}_{\alpha} \cap {\frak L}_{\bar 0}) \oplus ({\frak L}_{\alpha} \cap {\frak L}_{\bar 1})$. Hence, by denoting ${\frak L}_{\alpha, {\bar i}} := {\frak L}_{\alpha} \cap {\frak L}_{\bar i},$ for $\bar{i} \in \mathbb{Z}_2,$ we can write
\begin{equation}\label{separa}
{\frak L}_{\alpha} = {\frak L}_{\alpha, {\bar 0}} \oplus {\frak L}_{\alpha, {\bar 1}}
\end{equation}
for any $\alpha \in \Lambda \cup \{0\}$. From the above, $$\hbox{$H_{\bar 0} = {\frak L}_{0,{\bar 0}}$ \hspace{0.2cm} and \hspace{0.2cm} $H_{\bar 1} = {\frak L}_{0,{\bar 1}}$,}$$
and also
$$\hbox{${\frak L}_{\bar 0} = H_{\bar 0} \oplus (\bigoplus\limits_{\alpha \in \Lambda} {\frak L}_{\alpha,{\bar 0}})$ \hspace{0.2cm} and \hspace{0.2cm} ${\frak L}_{\bar 1} = H_{\bar 1} \oplus
(\bigoplus\limits_{\alpha \in \Lambda} {\frak L}_{\alpha,{\bar 1}})$.}$$
Taking into account this expression of ${\frak L}_{\bar 0}$, the direct character of the sum and the fact that $\alpha \neq 0$ for any $\alpha \in \Lambda$, we have that $H_{\bar 0}$ is a MASA of the Hom-Lie algebra ${\frak L}_{\bar 0}$. Hence ${\frak L}_{\bar 0}$ is a split regular Hom-Lie algebra (with respect to $H_{\bar 0}$) (cf. \cite{HomLie}) and ${\frak L}_{\bar 1}$ is (in a sense) a split anti-Lie triple system (with respect to $H_{\bar 0}$) (cf. \cite{Yotriple1}). For an easier notation, the mappings  $\phi|_{H_{\bar 0}}, \phi|_{H_{\bar 0}}^{-1} : H_{\bar 0} \to H_{\bar 0}$ will be denoted by $\phi$ and $\phi^{-1},$ respectively, when there is not possible confusion. The below lemma in an immediate consequence of super Hom-Jacobi identity.

\begin{lemma}\label{lema1}
For any $\alpha, \beta \in \Lambda \cup \{0\}$ the
following assertions hold.
\begin{enumerate}
\item[{\rm 1.}] If $\phi({\frak L}_{\alpha}) \neq 0$ then ${\alpha}\phi^{-1} \in \Lambda \cup \{0\}$ and $\phi({\frak L}_{\alpha}) = {\frak L}_{{\alpha}\phi^{-1}}.$ In a similar way, if $\phi^{-1}({\frak L}_{\alpha}) \neq 0$ then ${\alpha}\phi \in \Lambda \cup \{0\}$ and $\phi^{-1}({\frak L}_{\alpha}) = {\frak L}_{{\alpha}\phi}.$

\item[{\rm 2.}] If $[{\frak L}_{\alpha}, {\frak L}_{\beta}] \neq 0$ then $\alpha \phi^{-1}+\beta \phi^{-1} \in \Lambda \cup \{0\}$ and $[{\frak L}_{\alpha}, {\frak L}_{\beta}] \subset {\frak L}_{\alpha \phi^{-1}+\beta \phi^{-1}}.$
\end{enumerate}
\end{lemma}

\begin{proof}
1. For any $h_{\bar 0} \in H_{\bar 0}$ and $0 \neq v_{\alpha} \in {\frak L}_{\alpha},$ since $[h_{\bar 0}, v_{\alpha}] = \alpha(h_{\bar 0}) \phi(v_{\alpha}),$ by writing $h' := \phi(h_{\bar 0})$ we have
\begin{equation*}
[h',\phi(v_{\alpha})] = \phi([h_{\bar 0}, v_{\alpha}]) = \alpha(h_{\bar 0})\phi^2(v_{\alpha}) = \alpha\phi^{-1}(h')\phi(\phi(v_{\alpha})).
\end{equation*}
That is, $0 \neq \phi(v_{\alpha}) \in {\frak L}_{\alpha \phi^{-1}}, \alpha\phi^{-1} \in \Lambda \cup \{0\}$ and also
\begin{equation}\label{ro1}
\phi({\frak L}_{\alpha}) \subset {\frak L}_{\alpha \phi^{-1}}.
\end{equation}
Now, for any $h_{\bar 0} \in H_{\bar 0}$ and $0 \neq v_{\alpha} \in {\frak L}_{\alpha},$ as in the previous inclusion, for $h' := \phi^{-1}(h_{\bar 0})$ we get
\begin{equation*}
[h',\phi^{-1}(v_{\alpha})] = \phi^{-1}([h_{\bar 0}, v_{\alpha}]) = \alpha(h_{\bar 0})\phi^{-1}(\phi(v_{\alpha})) = \alpha\phi(h')\phi(\phi^{-1}(v_{\alpha})).
\end{equation*}
That is, $0 \neq \phi^{-1}(v_{\alpha}) \in {\frak L}_{\alpha \phi}, \alpha\phi \in \Lambda \cup \{0\}$ and conclude
\begin{equation}\label{ro2.25}
\phi^{-1}({\frak L}_{\alpha}) \subset {\frak L}_{\alpha \phi}.
\end{equation}
Now for any $x \in {\frak L}_{\alpha \phi^{-1}},$ we can write $x=\phi (\phi^{-1}(x))$ and since $\alpha\phi^{-1} \in \Lambda \cup \{0\}$ by Equation \eqref{ro2.25} we have $\phi^{-1}(x) \in {\frak L}_{\alpha}$, and we conclude ${\frak L}_{\alpha \phi^{-1}} \subset \phi({\frak
L}_{\alpha}).$
This fact together with Equation (\ref{ro1}) show $\phi({\frak L}_{\alpha}) = {\frak L}_{\alpha \phi^{-1}}$. In a similar way we also prove $\phi^{-1}({\frak L}_{{\alpha}}) = {\frak L}_{{\alpha}\phi}.$

2. For any $h_{\bar{0}} \in H_{\bar{0}}, v_{\alpha} \in {\frak L}_{\alpha}$ and $v_{\beta} \in {\frak L}_{\beta}$ such that $[v_{\alpha},v_{\beta}] \ne 0,$ by denoting $h' := \phi(h_{\bar{0}}),$ we get
\begin{eqnarray*}
\begin{split}
[h',[v_{\alpha},v_{\beta}]] &= -[[h_{\bar{0}},v_{\beta}],\phi(v_{\alpha})] + [[h_{\bar{0}}, v_{\alpha}],\phi(v_{\beta})]\\
&= (\alpha + \beta)(h_{\bar 0})\phi([v_{\alpha},v_{\beta}])\\
&= (\alpha + \beta)\phi^{-1}(h')\phi([v_{\alpha},v_{\beta}]).
\end{split}
\end{eqnarray*}
That is, $0 \neq [v_{\alpha},v_{\beta}] \in {\frak L}_{\alpha\phi^{-1} + \beta\phi^{-1}}$ and $\alpha\phi^{-1} + \beta\phi^{-1} \in \Lambda \cup \{0\}$.
\end{proof}

\begin{remark}\label{rmk1a}
From Lemma \ref{lema1} and Equation \eqref{separa} we can assert that for any $\alpha, \beta \in \Lambda \cup \{0\}$ and $\bar{i},\bar{j} \in \mathbb{Z}_2$:
\begin{enumerate}
\item[{\rm 1.}] If $\phi({\frak L}_{{\alpha},\bar{i}}) \neq 0$ then ${\alpha}\phi^{-1} \in \Lambda \cup \{0\}$ and $\phi({\frak L}_{{\alpha},\bar{i}}) = {\frak L}_{{\alpha}\phi^{-1},\bar{i}},$

\item[{\rm 2.}] If $\phi^{-1}({\frak L}_{{\alpha},\bar{i}}) \neq 0$ then ${\alpha}\phi \in \Lambda \cup \{0\}$ and $\phi^{-1}({\frak L}_{{\alpha},\bar{i}}) = {\frak L}_{{\alpha}\phi,\bar{i}},$

\item[{\rm 3.}] If $[{\frak L}_{\alpha,\bar{i}}, {\frak L}_{\beta,\bar{j}}] \neq 0$ then $\alpha \phi^{-1}+\beta \phi^{-1} \in \Lambda \cup \{0\}$ and $[{\frak L}_{\alpha,\bar{i}}, {\frak L}_{\beta,\bar{j}}] \subset {\frak L}_{\alpha \phi^{-1}+\beta \phi^{-1},\bar{i}+\bar{j}},$
\end{enumerate}
\end{remark}

\begin{lemma}\label{lema2}
If $\alpha \in \Lambda$ then $\alpha\phi^z \in \Lambda$ for any $z \in {\mathbb Z}$.
\end{lemma}

\begin{proof}
Consequence of Lemma \ref{lema1}-1.
\end{proof}

\begin{remark}\label{remarkr}
From Lemma \ref{lema2}, Remark \ref{rmk1a} and Equation \eqref{separa} we can assert if ${\frak L}_{\alpha,\bar{i}} \neq 0,$ with ${\bar i} \in \mathbb{Z}_2,$ then ${\frak L}_{\alpha\phi^z,\bar{i}} \neq 0$ for any $z \in {\mathbb Z}.$
\end{remark}

%%%%%%%%%%%%%%%%%%%%%%%%%%%%%%%%%%%%%%%%%%%%%%%%%%%%%%%%%%%%%%
%%%%%%%%%%%%%%%%%%%%%%%%%%%%%%%%%%%%%%%%%%%%%%%%%%%%%%%%%%%%%%
\section{Connections of roots techniques}
%%%%%%%%%%%%%%%%%%%%%%%%%%%%%%%%%%%%%%%%%%%%%%%%%%%%%%%%%%%%%%
%%%%%%%%%%%%%%%%%%%%%%%%%%%%%%%%%%%%%%%%%%%%%%%%%%%%%%%%%%%%%%

In the following, ${\frak L}$ denotes a split regular Hom-Lie superalgebra and $${\frak L} = {\frak L}_0 \oplus \Bigl(\bigoplus\limits_{\alpha \in \Lambda}{\frak L}_{\alpha}\Bigr)$$ the corresponding  root spaces decomposition. Given a linear functional $\alpha: H_{\bar 0} \to {\hu K}$, we denote by $-\alpha: H_{\bar 0} \to {\hu K}$ the element in $(H_{\bar 0})^*$ defined by $(-\alpha)(h_{\bar 0}):=-\alpha(h_{\bar 0})$ for all $h_{\bar 0} \in H_{\bar 0}$. We also denote by
$-\Lambda := \{-\alpha: \alpha \in \Lambda\}$ and $$\pm \Lambda := \Lambda \cup -\Lambda.$$

\begin{definition}\label{connection}\rm
Let $\alpha$ and $\beta$ be two elements in $\Lambda$. We shall
say that $\alpha$ is {\em connected} to  $\beta$ if there exists $\{\alpha_1,\alpha_2,\dots,\alpha_{k}\}\subset \pm \Lambda$ such that:
\begin{enumerate}
\item[] If $k=1$,
\begin{enumerate}
\item[1.] $\alpha_1 \in \{\alpha\phi^{-n}: n \in {\mathbb N}_0\}
\cap \{\pm \beta\phi^{-m}: m \in {\mathbb N}_0\}.$
\end{enumerate}

\item[] If $k \geq 2$,
\begin{enumerate}
\item[1.] $\alpha_1 \in \{\alpha \phi^{-n}: n \in {\mathbb N}_0\}$.
\item[2.] $\alpha_1 \phi^{-1} + \alpha_2\phi^{-1} \in \pm\Lambda,$

$\alpha_1 \phi^{-2}+ \alpha_2\phi^{-2}+ \alpha_3\phi^{-1} \in \pm\Lambda,$

$\alpha_1 \phi^{-3} + \alpha_2 \phi^{-3}+ \alpha_3 \phi^{-2}+
\alpha_4\phi^{-1} \in \pm\Lambda,$

\hspace{3.5cm} $\vdots$

$\alpha_1\phi^{-i}+ \alpha_2\phi^{-i} + \alpha_3\phi^{-i+1}+
\cdots +\alpha_{i+1}\phi^{-1} \in \pm\Lambda,$

\hspace{3.5cm} $\vdots$

$\alpha_1\phi^{-k+2}+ \alpha_2\phi^{-k+2} +
\alpha_3\phi^{-k+3}+ \cdots +\alpha_i\phi^{-k+i}+ \cdots
+\alpha_{k-1}\phi^{-1} \in \pm\Lambda.$

\item[3.] $\alpha_1\phi^{-k+1}+ \alpha_2\phi^{-k+1} +
\alpha_3\phi^{-k+2}+ \cdots + \alpha_i\phi^{-k+i-1}+
\cdots+\alpha_k\phi^{-1}\in \{\pm \beta\phi^{-m}: m \in
{\mathbb N}_0\}.$
\end{enumerate}
\end{enumerate}
We shall also say that $\{\alpha_1,\dots,\alpha_k\}$ is a {\it connection} from $\alpha$ to $\beta$.
\end{definition}

\noindent Observe that the above concept of connection of nonzero roots from $\alpha$ to $\beta$, given in Definition \ref{connection},  for the case $k=1$ is equivalent to the fact $\beta = \epsilon \alpha\phi^{z}$ for some $z \in {\mathbb Z}$ and $\epsilon \in \{1,-1\}$. Our next goal is to show that the connection relation is of equivalence. We begin with a couple of previous lemmas being their proof analogous to the ones of Lemma 2.1 and Lemma 2.2 in \cite{HomLie}.

\begin{lemma}\label{elprimero}
For any $\alpha \in \Lambda, \alpha\phi^{z_1}$ is connected to $ \alpha\phi^{z_2}$ for any $z_1, z_2 \in {\mathbb Z},$ and also to $- \alpha\phi^{z_2} $ in case $- \alpha \in \Lambda$.
\end{lemma}
%\begin{proof}
%By Lemma \ref{lema2} we have that $\alpha\phi^{z_1}, \alpha\phi^{z_2}\in \Lambda$. Let $z := \min\{z_1,z_2\},$ then we have that $\{\alpha\phi^{z}\}$ is a connection from $\alpha\phi^{z_1}$ to $ \alpha\phi^{z_2}$ and to $ -\alpha\phi^{z_2}$ in case $ -\alpha \in \Lambda$.
%\end{proof}

\begin{lemma}\label{transi1}
Let $\{\alpha_1,\dots,\alpha_k\}$ be a connection from $\alpha$
to $\beta$. The following assertions hold.
\begin{enumerate}
\item[\rm 1.] If $\alpha_1=\alpha\phi^{-n}, n \in {\mathbb N}_0,$ then for any $r \in {\mathbb N}_0$ such that $r \geq n$, there exists a connection $\{\bar \alpha_1,\dots,\bar \alpha_k\}$ from $\alpha$ to $\beta$ such that $\bar \alpha_1=\alpha\phi^{-r}$.

\item[\rm 2.] If $\alpha_1 = \epsilon \beta\phi^{-m}$ in case
$k=1$ or $$\alpha_1\phi^{-k+1} + \alpha_2\phi^{-k+1} + \alpha_3\phi^{-k+2} + \cdots + \alpha_k\phi^{-1} = \epsilon
\beta\phi^{-m} $$ in case $k\geq 2$, with $m \in {\mathbb N}_0$ and $\epsilon \in \{1,-1\},$ then for any $r \in {\mathbb N}_0$ such that $r \geq m$, there exists a connection $\{\bar \alpha_1,\dots,\bar \alpha_k\}$ from $\alpha$ to $\beta$ such that $\bar \alpha_1 = \epsilon\beta\phi^{-r}$ in case $k=1$ or
$$\bar \alpha_1\phi^{-k+1} + \bar \alpha_2\phi^{-k+1} + \bar \alpha_3\phi^{-k+2} + \cdots + \bar \alpha_k\phi^{-1}= \epsilon \beta \phi^{-r}$$ in case $k \geq 2$.
\end{enumerate}
\end{lemma}

%\begin{proof}
%{\rm 1.} By Lemma \ref{lema2} we have $\{\alpha_1\phi^{n-r},\dots,\alpha_k\phi^{n-r}\} \subset \pm \Lambda$. Define $\bar \alpha_i := \alpha_i\phi^{n-r},$ $i \in \{1,\dots,k\}$, then Lemma \ref{lema2} allows us to verify that $\{\bar \alpha_1,\dots,\bar \alpha_k\}$ is a connection from $\alpha$ to $\beta$ which clearly satisfies $\bar \alpha_1 = (\alpha \phi^{-n})\phi^{n-r} = \alpha\phi^{-r}$.

%{\rm 2.} Lemma \ref{lema2} allows us to assert that $\{\alpha_1\phi^{m-r},\dots,\alpha_k\phi^{m-r}\} \subset \pm\Lambda.$ Define now $\bar \alpha_i := \alpha_i\phi^{m-r}, i \in \{1,\dots,k\}$, then Lemma \ref{lema2} implies $\{\bar \alpha_1,\dots,\bar \alpha_k\}$ is a connection from $\alpha$ to $\beta$. It is clear that $\bar \alpha_1 = \epsilon \beta\phi^{-r}$ in case $k=1$, and also $$\bar \alpha_1\phi^{-k+1} + \bar \alpha_2\phi^{-k+1} + \bar \alpha_3\phi^{-k+2} + \cdots + \bar \alpha_k \phi^{-1}=$$
%$$(\alpha_1\phi^{-k+1} + \alpha_2\phi^{-k+1} + \alpha_3\phi^{-k+2} + \cdots + \alpha_k\phi^{-1})\phi^{m-r} = \epsilon \beta \phi^{-r}$$ in case $k\geq 2$. \black
%\end{proof}

\noindent The proof of the next result follows the ideas of \cite{HomLie} for split Hom-Lie algebras. For the sake of completeness we add an sketch of the same.

\begin{proposition}\label{pro1}
The relation $\sim$ in $\Lambda$ defined by $\alpha \sim \beta$ if and only if  $\alpha$ is connected to $\beta$ is an equivalence relation.
\end{proposition}

\begin{proof}
By Lemma \ref{elprimero} the relation $\sim$ is reflexive. The symmetric character of $\sim$ is consequence of consider $\alpha \sim \beta$ with a connection $\{\alpha_1,\alpha_2,\alpha_3,\dots,\alpha_{k-1},\alpha_k\} \subset \pm \Lambda$ from $\alpha$ to $\beta$. If $k=1$ we get $\beta \sim \alpha$ with the connection $\{\epsilon \alpha_1\}.$ If $k \geq 2$, we conclude $\beta \sim \alpha$ taking the connection
$$\{\beta\phi^{-m}, - \alpha_k \phi^{-1}, -\alpha_{k-1}\phi^{-3}, -\alpha_{k-2} \phi^{-5},\dots, -\alpha_{k-i}\phi^{-2i-1},\dots, -\alpha_2 \phi^{-2k+3}\} \subset \pm \Lambda$$ in case
$\alpha_1\phi^{-k+1} + \alpha_2\phi^{-k+1} +
\alpha_3\phi^{-k+2} + \cdots + \alpha_{k-i+1}\phi^{-i}+
\cdots + \alpha_k\phi^{-1}= \beta\phi^{-m},$ and the connection $$\bigl\{\beta\phi^{-m}, \alpha_k\phi^{-1}, \alpha_{k-1}\phi^{-3}, \alpha_{k-2}\phi^{-5},\dots, \alpha_{k-i}\phi^{-2i-1},\dots,  \alpha_2\phi^{-2k+3}\bigr\} \pm \Lambda$$ in case $\alpha_1\phi^{-k+1} + \alpha_2\phi^{-k+1} + \alpha_3\phi^{-k+2}+ \cdots + \alpha_{k-i+1}\phi^{-i} + \cdots + \alpha_k\phi^{-1}= -\beta\phi^{-m},$ for some $m \in {\mathbb N}_0.$

Finally, to verify that $\sim$ is transitive we suppose $\alpha \sim \beta$ with connection $\{\alpha_1,\dots,\alpha_k\}$ such that
\begin{equation}\label{1001}
\hbox{$\alpha_1=\epsilon \beta\phi^{-m}$ if $k=1$}
\end{equation}
for some $m \in {\mathbb N}_0$, $\epsilon \in \{1,-1\},$ or
\begin{equation}\label{unsigno}
\hbox{$\alpha_1\phi^{-k+1} + \alpha_2\phi^{-k+1} + \alpha_3\phi^{-k+2} + \cdots + \alpha_k\phi^{-1} = \epsilon
\beta\phi^{-m}$ if $k \geq 2$,}
\end{equation}
for some $r \in {\mathbb N}_0,$ and suppose $\beta \sim \gamma$ with connection $\{h_1,\dots,h_p\}.$ If $p=1$, $\{\alpha_1,\dots,\alpha_k\}$ is also a connection from $\alpha$ to $\gamma$. If $p\geq 2$, $\{\alpha_1,\dots,\alpha_k, h_2,\dots, h_p\}$ is a connection from $\alpha$ to $\gamma$ if $\epsilon =1$ in Equations \eqref{1001} or \eqref{unsigno}, and  $\{\alpha_1,\dots, \alpha_k,-h_2,\dots,-h_p\}$ is the referred connection if $\epsilon = -1$ in Equations \eqref{1001} or \eqref{unsigno}. Summarizing, the connection relation is also transitive and so it is an equivalence relation.
\end{proof}

%%%%%%%%%%%%%%%%%%%%%%%%%%%%%%%%%%%%%%%%%%%%%%%%%%%%%%%%%%%%%%
%%%%%%%%%%%%%%%%%%%%%%%%%%%%%%%%%%%%%%%%%%%%%%%%%%%%%%%%%%%%%%
\section{Decompositions}
%%%%%%%%%%%%%%%%%%%%%%%%%%%%%%%%%%%%%%%%%%%%%%%%%%%%%%%%%%%%%%
%%%%%%%%%%%%%%%%%%%%%%%%%%%%%%%%%%%%%%%%%%%%%%%%%%%%%%%%%%%%%%

By Proposition \ref{pro1} the connection relation is an
equivalence relation in $\Lambda$ and so we can consider the
quotient set $$\Lambda/\sim := \{[\alpha] : \alpha \in \Lambda \},$$ becoming $[\alpha]$ the set of nonzero roots ${\frak L}$ which are connected to $\alpha$. Our next goal in this section is to associate an (adequate)  ideal ${I}_{{[\alpha]}}$ of the split regular Hom-Lie superalgebra ${\frak L}$ to any $[\alpha]$ of $\Lambda/\sim$. For $[\alpha]$, with $\alpha \in \Lambda$, we define
$$H_{[\alpha]} := span_{\hu K}\bigl\{[{\frak L}_{\beta},{\frak L}_{\delta}] : \beta,\delta \in [\alpha]\bigr\} \cap {\frak L}_0.$$ Applying Lemma \ref{lema1}-2 we obtain
$$H_{[\alpha]} := span_{\hu K}\bigl\{[{\frak L}_{\beta},{\frak L}_{-\beta}] : \beta \in [\alpha]\bigr\} =$$
$$\Bigl(\sum\limits_{\beta \in [\alpha]}\bigl([{\frak L}_{\beta,{\bar 0}},{\frak L}_{-\beta,{\bar 0}}] + [{\frak L}_{\beta,{\bar 1}},{\frak L}_{-\beta,{\bar 1}}]\bigr)\Bigr) \oplus \Bigl(\sum\limits_{\beta \in [\alpha]}\bigl([{\frak L}_{\beta,{\bar 0}},{\frak L}_{-\beta,{\bar 1}}] + [{\frak L}_{\beta,{\bar 1}},{\frak L}_{-\beta,{\bar 0}}]\bigr)\Bigl) \subset H_{\bar 0} \oplus H_{\bar 1},$$ last equality being consequence of Equation \eqref{separa}, and we also define $$V_{[\alpha]} := \bigoplus\limits_{\beta \in [\alpha]}{\frak L}_{\beta} = (\bigoplus\limits_{\beta \in [\alpha]} {\frak L}_{\beta,{\bar 0}}) \oplus (\bigoplus\limits_{\beta \in [\alpha]} {\frak L}_{\beta, {\bar 1}}).$$ We denote by ${\frak L}_{[\alpha]}$ the following (graded) subspace of ${\frak L}$, $${\frak L}_{[\alpha]}:=H_{[\alpha]} \oplus V_{[\alpha]}.$$

\begin{proposition}\label{pro2}
For any $[\alpha] \in \Lambda/\sim$, the linear subspace ${\frak L}_{[\alpha]}$ is a graded subalgebra of ${\frak L}$.
\end{proposition}

\begin{proof}
First, we have to check that ${\frak L}_{{[\alpha]}}$ satisfies
$[{\frak L}_{[\alpha]}, {\frak L}_{[\alpha]}] \subset {\frak L}_{[\alpha]}$. Since $H_{[\alpha]} \subset {\frak L}_0 =H$, then
$[H_{[\alpha]}, H_{[\alpha]}]=0$ and we have
\begin{equation}\label{cero}
\bigl[H_{[\alpha]} \oplus V_{[\alpha]}, H_{[\alpha]} \oplus V_{[\alpha]}\bigr] \subset [H_{[\alpha]}, V_{[\alpha]}] +[V_{[\alpha]}, V_{[\alpha]}].
\end{equation}
Let us consider the first summand in Equation \eqref{cero}. Given $\beta \in [\alpha]$ we have $[H_{[\alpha]}, {\frak
L}_{\beta}] \subset {\frak L}_{\beta \phi^{-1}}$, being $\beta
\phi^{-1} \in [\alpha]$ by Lemma \ref{elprimero}. Hence
\begin{equation}\label{ceroo}
[H_{[\alpha]}, {\frak L}_{\beta}] \subset V_{[\alpha]}.
\end{equation}
Consider now the second summand in Equation \eqref{cero}, that is, $[V_{[\alpha]}, V_{[\alpha]}]$. Given $\beta, \gamma \in
[\alpha]$ such that $[{\frak L}_{\beta}, {\frak L}_{\gamma}]
\neq 0$, if $\gamma = -\beta$ then clearly $[{\frak L}_{\beta }, {\frak L}_{\gamma}] = [{\frak L}_{\beta}, {\frak L}_{-\beta}] \subset H_{[\alpha]}.$ Suppose $\gamma \neq -\beta$. Taking into account that the fact $[{{\frak L}}_{\beta},  {{\frak L}}_{\gamma}] \neq 0$ together with Lemma \ref{lema1}-2  ensure $\beta \phi^{-1}+\gamma \phi^{-1}\in \Lambda$, we have that $\{\beta,\gamma \}$ is a connection from $\beta $
to $\beta \phi^{-1}+ \gamma \phi^{-1} $. The transitivity of
$\sim$ gives now that $\beta \phi^{-1}+ \gamma \phi^{-1} \in
[\alpha]$ and so $[{\frak L}_{\beta}, {\frak L}_{\gamma}] \subset {\frak L}_{\beta \phi^{-1} + \gamma \phi^{-1}} \subset V_{[\alpha]}$. Hence $[\bigoplus\limits_{\beta \in [\alpha]} {\frak L}_{\beta}, \bigoplus \limits_{\beta \in [\alpha]} {\frak L}_{\beta}] \subset H_{[\alpha]} \oplus V_{[\alpha]}$. That is,
\begin{equation}\label{eq0.5}
[V_{[\alpha]}, V_{[\alpha]}] \subset {\frak L}_{[\alpha]}.
\end{equation}
From Equations \eqref{cero}-\eqref{eq0.5} we get
$[{\frak L}_{[\alpha]}, {\frak L}_{[\alpha]}] = [H_{[\alpha]} \oplus V_{[\alpha]},H_{[\alpha]} \oplus V_{[\alpha]}] \subset {\frak L}_{[\alpha]}.$

Second, we have to verify that $\phi({\frak L}_{[\alpha]}) = {\frak L}_{[\alpha]}$. But this is a direct consequence of Lemma \ref{lema1}-1 and Lemma \ref{elprimero}.
\end{proof}

\begin{proposition}\label{pro9}
If $[\alpha] \neq [\gamma]$ then $[{\frak L}_{[\alpha]} , {\frak L}_{[{\gamma}]}]=0$.
\end{proposition}

\begin{proof}
We have
\begin{equation}\label{cuatro}
\bigl[H_{[\alpha]} \oplus V_{[\alpha]}, H_{[\gamma]} \oplus V_{[\gamma]}\bigr] \subset [H_{[\alpha]},V_{[\gamma]}] + [V_{[\alpha]}, H_{[\gamma]}] + [V_{[\alpha]},V_{[\gamma]}].
\end{equation}

Consider the above third summand $[V_{[\alpha]}, V_{[\gamma]}]$ and suppose there exist $\alpha_1 \in [\alpha]$ and $\gamma_1 \in [{\gamma}]$ such that $[{\frak L}_{\alpha_1}, {\frak L}_{\gamma_1}] \neq 0$. As necessarily $\alpha_1 \neq -\gamma_1$, then $\alpha_1 \phi^{-1} + \gamma_1 \phi^{-1} \in \Lambda$. So $\{\alpha_1, \gamma_1, -\alpha_1\phi^{-1} \}$ is a connection between $\alpha_1$ and $\gamma_1$. By the transitivity of the connection relation we have $\alpha \in [\gamma]$, a contradiction. Hence $[{\frak L}_{\alpha_1}, {\frak L}_{\gamma_1}]= 0$ and so
\begin{equation}\label{nueve}
[V_{[\alpha]}, V_{[\gamma]}]=0.
\end{equation}

Consider now the first summand $[H_{[\alpha]},V_{[\gamma]}]$ in Equation (\ref{cuatro}) and suppose there exist $\alpha_1 \in [\alpha]$ and $\gamma_1 \in [\gamma]$ such that $[[{\frak L}_{\alpha_1},{\frak L}_{-\alpha_1}],{\frak L}_{\gamma_1}] \neq 0$. By Lemma \ref{lema1}-1 we have ${\frak L}_{\gamma_1} = \phi({{\frak L}}_{\gamma_1\phi})$, so we obtain $\bigl[[{\frak L}_{\alpha_1},{\frak L}_{-\alpha_1}],\phi({\frak L}_{\gamma_1\phi})\bigr] \neq 0.$ Using the grading of ${\frak L}_{\alpha_1}$ we get $$\bigl[[{\frak L}_{\alpha_1,\bar{0}},{\frak L}_{-\alpha_1}], \phi({{\frak L}}_{\gamma_1\phi})\bigr] + \bigl[[{\frak L}_{\alpha_1,\bar{1}},{\frak L}_{-\alpha_1}], \phi({{\frak L}}_{\gamma_1\phi})\bigr] \neq 0.$$ Therefore $\bigl[[{\frak L}_{\alpha_1,\bar{0}},{\frak L}_{-\alpha_1}], \phi({{\frak L}}_{\gamma_1\phi})\bigr] \neq 0$ or $\bigl[[{\frak L}_{\alpha_1,\bar{1}},{\frak L}_{-\alpha_1}], \phi({{\frak L}}_{\gamma_1\phi})\bigr] \neq 0.$
In the first case, by super Hom-Jacobi identity we get either $[{\frak L}_{-\alpha_1},{\frak L}_{\gamma_1\phi}] \neq 0$ or $[{\frak L}_{\gamma_1\phi},{\frak L}_{\alpha_1, {\bar 0}}] \neq 0$. From here $[V_{[\alpha]},V_{[\gamma]}] \neq 0$ in any case, what contradicts Equation \eqref{nueve}. In the second case we have $0\bigl[[{\frak L}_{\alpha_1,{\bar 1}},{\frak L}_{-\alpha_1,{\bar 0}}],{\frak L}_{\gamma_1\phi}\bigr] + \bigl[[{\frak L}_{\alpha_1,{\bar 1}},{\frak L}_{-\alpha_1,{\bar 1}}],{\frak L}_{\gamma_1\phi}\bigr] \neq 0$ and so either $\bigl[[{\frak L}_{\alpha_1, {\bar 1}},{\frak L}_{-\alpha_1,{\bar 0}}],{\frak L}_{\gamma_1\phi}\bigr] \neq 0$ or $\bigl[[{\frak L}_{\alpha_1,{\bar 1}},{\frak L}_{-\alpha_1,{\bar 1}}],{\frak L}_{\gamma_1\phi}\bigr] \neq 0.$ By applying super Hom-Jacobi identity to any of these summands we have as above that necessarily $[V_{[\alpha]},V_{[\gamma]}] \neq 0$ what contradicts again Equation \eqref{nueve}.
From here $[H_{[\alpha]},V_{[\gamma]}] = 0$.

In a similar way we get $[V_{[\alpha]},H_{[\gamma]}] = 0$ and we conclude, together with Equations (\ref{cuatro}) and (\ref{nueve}) that $[{\frak L}_{[\alpha]}, {\frak L}_{[\beta]}] = 0.$
\end{proof}

\begin{theorem}\label{teo1}
Let ${\frak L}$ be a split regular Hom-Lie superalgebra. The following assertions hold:
\begin{enumerate}
\item[{\rm 1.}] For any $[\alpha] \in \Lambda/\sim$, the (graded) subalgebra $${\frak L}_{[\alpha]} = H_{[\alpha]} \oplus V_{[\alpha]}$$ of ${\frak L}$ associated to $[\alpha]$ is an ideal of ${\frak L}$.

\item[{\rm 2.}] If ${\frak L}$ is simple then there exists a connection from $\alpha$ to $\beta$ for any $\alpha, \beta \in \Lambda$ and $H = \sum\limits_{\alpha \in \Lambda}[{\frak L}_{\alpha}, {\frak L}_{-\alpha}]$.
\end{enumerate}
\end{theorem}

\begin{proof}
1. Since $[{\frak L}_{[\alpha]}, H] \subset {\frak L}_{[\alpha]}$ we have by Proposition \ref{pro2} and Proposition \ref{pro9} that $$[{\frak L}_{[\alpha]}, {\frak L}] = [{\frak L}_{[\alpha]}, H \oplus (\bigoplus\limits_{\beta \in [\alpha]}{\frak L}_{{\beta}}) \oplus (\bigoplus\limits _{{\gamma} \notin [\alpha]}{{\frak L}}_{{\gamma}})]\subset {\frak L}_{[\alpha]}.
$$
We also have by Proposition \ref{pro2} that
$\phi({\frak L}_{[\alpha]}) = {\frak L}_{[\alpha]}$, we conclude ${\frak L}_{[\alpha]}$ is an ideal of ${\frak L}$.

\medskip

2. The simplicity of ${\frak L}$ implies ${\frak L}_{[\alpha]}={\frak L}$. From here, it is clear that $[\alpha] = \Lambda$ and  $H=\sum\limits_{\alpha \in \Lambda}[{\frak L}_{\alpha}, {\frak L}_{-\alpha}]$.
\end{proof}

\begin{theorem}\label{teo2}
A split regular Hom-Lie superalgebra ${\frak L}$ can be represented as $${\frak L} = U + \sum\limits_{[\alpha] \in \Lambda/\sim}{\frak L}_{[\alpha]},$$ where $U$ is a linear complement  in $H$ of $span_{\mathbb K}\{[{\frak L}_{\alpha},{\frak L}_{-\alpha}] : \alpha \in \Lambda\}$ and any ${\frak L}_{[\alpha]}$ is one of the ideals of ${\frak L}$ described in Theorem \ref{teo1}-1, satisfying $[{\frak L}_{[\alpha]}, {\frak L}_{[\gamma]}] = 0$ if $[\alpha] \neq [\gamma].$
\end{theorem}

\begin{proof}
We have ${\frak L}_{[\alpha]}$ is well defined and, by Theorem
\ref{teo1}-1, an ideal of $\frak L$, being clear that $${\frak L} = H \oplus (\bigoplus\limits_{\alpha \in \Lambda} {\frak L}_{\alpha}) = U + \sum\limits_{[\alpha] \in \Lambda/\sim} {\frak L}_{[\alpha]}.$$ Finally from Proposition \ref{pro9} we obtain $[{\frak L}_{[\alpha]}, {\frak L}_{[\gamma]}] = 0$ if $[\alpha] \neq [\gamma].$
\end{proof}

\noindent Let us denote by ${\mathcal Z}({\frak L}) := \bigl\{v \in {\frak L} : [v, {\frak L}] = 0\bigr\}$ the {\it center} of ${\frak L}$.

\begin{corollary}\label{coro1}
If ${\mathcal Z}({\frak L}) = 0$ and $H = \sum_{\alpha \in \Lambda}[{\frak L}_{\alpha}, {\frak L}_{-\alpha}]$ then ${\frak L}$ is the direct sum of the ideals given in Theorem \ref{teo1}, $${\frak L} = \bigoplus\limits_{[\alpha] \in \Lambda/\sim} {\frak L}_{[\alpha]},$$ being $[{\frak L}_{[\alpha]}, {\frak L}_{[\gamma]}]=0$ if $[\alpha] \neq [\gamma].$
\end{corollary}

\begin{proof}
Since $H = \sum_{\alpha \in \Lambda}[{\frak L}_{\alpha}, {\frak L}_{-\alpha}]$ we get ${\frak L} = \sum_{[\alpha] \in \Lambda/\sim} {\frak L}_{[\alpha]}$. Finally, the direct character of the sum can be followed from the facts $[{\frak L}_{[\alpha]}, {\frak L}_{[\gamma]}] = 0$, if $[\alpha] \neq [\gamma]$, and ${\mathcal Z}({\frak L}) = 0$.
\end{proof}

\begin{example}\label{example2}
Let us consider the split Hom-Lie superalgebra ${{\frak L}}$  in Example \ref{example1}. We have that $$\Lambda=\{\alpha\} \cup \{ \pm \beta_n, \pm \gamma_n ; n \in {\mathbb N}, n \geq 2 \}.$$ It is straightforward to verify that $$\Lambda / \sim = \{ [ \alpha] \} \cup \{ [ \beta _n] : n \in {\mathbb N}, n \geq 2 \},$$ 
where $[\alpha]=\{\alpha\}$ and any $[\beta_n]=\{\pm \beta_n, \pm \gamma_n \}$. 
Indeed, for checking for instance $\beta_n \sim \gamma_n$, observe that the fact $\beta_n \phi^{-1} + \gamma_n \phi^{-1}= - \gamma_n$  gives us  that the set $\{\beta_n, \gamma_n \}$ is a connection from $\beta_n$ to  $\gamma_n.$

From here, the  results in this section allow us to consider  the family of ideals $$\hbox{${\frak L}_{[\alpha]}:={\frak L}_{\alpha},$ and }$$
$$\hbox{${\frak L}_{[\beta_n]}:=\langle h_n\rangle \oplus {\frak L}_{\beta_n} \oplus {\frak L}_{-\beta_n}\oplus {\frak L}_{\gamma_n}\oplus {\frak L}_{-\gamma_n}$, for $n \geq 2$};$$ and the decomposition of ${{ \frak L}}$ as the sum of the ideals

\begin{equation}\label{sep2}
{{\frak L}}= \langle e_2, e_3 \rangle \oplus { \frak L}_{[\alpha]} \oplus ( \bigoplus\limits_{n \in {\mathbb N}, n \geq 2} {\frak L}_{[\beta_n]}).
\end{equation}

\end{example}

%%%%%%%%%%%%%%%%%%%%%%%%%%%%%%%%%%%%%%%%%%%%%%%%%%%%%%%%%%%%%%
%%%%%%%%%%%%%%%%%%%%%%%%%%%%%%%%%%%%%%%%%%%%%%%%%%%%%%%%%%%%%%
\section{The simple components}
%%%%%%%%%%%%%%%%%%%%%%%%%%%%%%%%%%%%%%%%%%%%%%%%%%%%%%%%%%%%%%
%%%%%%%%%%%%%%%%%%%%%%%%%%%%%%%%%%%%%%%%%%%%%%%%%%%%%%%%%%%%%%

In this section we are interested in studying under which
conditions ${\frak L}$ decomposes as the direct sum of the family of its simple ideals. We recall that a roots system $\Lambda$ of a split regular Hom-Lie superalgebra ${\frak L}$ is called {\it symmetric} if it satisfies that $\alpha \in \Lambda$ implies $-\alpha \in \Lambda$. From now on we will suppose $\Lambda$ is symmetric.

\begin{lemma}\label{lema4}
If $I$ is an ideal of ${{\frak L}}$ such that $I \subset H$, then $I \subset {\mathcal Z}({{\frak L}}).$
\end{lemma}

\begin{proof}
Consequence of $[I,H] \subset [H,H]=0$ and $[I , \bigoplus \limits_{\alpha \in {\Lambda}} {{\frak L}}_{\alpha}] \subset (\bigoplus \limits_{\alpha \in {\Lambda}} {{\frak L}}_{\alpha}) \cap H=0.$
\end{proof}

\noindent Observe that the grading of $I$ and Equation \eqref{separa} let us write
\begin{equation}\label{ideal}
I=I_{\bar 0} \oplus I_{\bar 1}= \Bigl(\bigl(I_{\bar 0} \cap H_{\bar 0}\bigr) \oplus \bigl(\bigoplus\limits_{\alpha \in \Lambda} (I_{\bar 0} \cap {\frak L}_{\alpha, {\bar 0}})\bigr)\Bigr) \oplus \Bigl(\bigl(I_{\bar 1}\cap H_{\bar 1}\bigr) \oplus \bigl(\bigoplus\limits_{\alpha \in \Lambda} (I_{\bar 1} \cap {\frak L}_{\alpha, {\bar 1}})\bigr)\Bigr).
\end{equation}

\noindent Next lemmas have a proof similar to the ones of Lemma 4.2 and Lemma 4.3 in \cite{HomLie}.

\begin{lemma}\label{navi}
For any $\alpha, \beta \in \Lambda$ with $\alpha \neq \beta$ there exists $h_0 \in H_{\bar 0}$ such that $\alpha (h_0)\neq 0$ and $\alpha (h_0) \neq \beta (h_0)$.
\end{lemma}
%\begin{proof}
%As $\alpha \neq \beta$, there exists $h\in H$ such that $\alpha(h)\neq \beta(h)$. If $\alpha(h)\neq 0$ we have finished, so let us suppose $\alpha(h)= 0$  what implies $\beta(h) \neq 0$. Since $\alpha \neq 0$, we can  fix some $h^{\prime}\in H$ such that $\alpha(h^{\prime})\neq 0$. We can distinguish two cases, in the first one $\alpha(h^{\prime})\neq \beta(h^{\prime})$ and in the second one $\alpha(h^{\prime})= \beta(h^{\prime})$. Then we have that by taking $h_0:=h^{\prime}$ in the first case and $h_0:=h+h^{\prime}$ in the second one we complete the proof.
%\end{proof}

\begin{lemma}\label{navi2}
If $I$ is an ideal of ${\frak L}$ and $x = h + \sum_{j=1}^n  v_{\alpha_j} \in I$, with $h \in H$, $v_{\alpha_j}\in {\frak L}_{\alpha_j}$ and ${\alpha_j} \neq \alpha_k$ if $j \neq k$. Then any $v_{\alpha_j}\in I$.
\end{lemma}

\noindent Let us introduce the concepts of root-multiplicativity and maximal length in the framework of split Hom-Lie superalgebras in a similar way to the ones for split Hom-Lie algebras, split Lie algebras, split color algebras, split Leibniz algebras and split Lie superalgebras (see \cite{HomLie, YoLi, Yocolor, YoLeibniz, Nosalg} for these notions and examples). We denote by $\Lambda_{\bar 0} := \{\alpha \in \Lambda : {\frak L}_{\alpha, {\bar 0}} \neq 0\}$ and by $\Lambda_{\bar 1} := \{\alpha \in \Lambda: {\frak L}_{\alpha, {\bar 1}} \neq 0\}$ (see Equation \eqref{separa}). So $\Lambda = \Lambda_{\bar 0} \cup \Lambda_{\bar 1}$, being a non necessarily disjoint union.

\begin{definition}
We say that a split regular Hom-Lie superalgebra ${\frak L}$ is {\it root-multiplicative} if given $\alpha \in \Lambda_{\bar i},$ $\beta \in \Lambda_{\bar j},$ for $\bar{i},\bar{j} \in \mathbb{Z}_2,$ such that $\alpha \phi^{-1}+\beta \phi^{-1}\in \Lambda_{\bar{i}+\bar{j}},$ then $[{\frak L}_{\alpha,\bar{i}} ,{\frak L}_{\beta,\bar{j}}]\neq 0.$
\end{definition}

\begin{definition}
It is said that a split regular Hom-Lie superalgebra ${\frak L}$ is of {\it maximal length} if $\dim {\frak L}_{\alpha,\bar{i}} \in \{0,1\}$ for any $\alpha \in \Lambda$ with $\bar{i} \in \mathbb{Z}_2.$
\end{definition}

\noindent Observe that if ${\frak L}$ is of maximal length, then the grading of $I$ and Equation (\ref{separa}) allow us to assert that given any nonzero ideal $I$ of ${\frak L}$ then
\begin{equation}\label{max}
I = \Bigl(\bigl(I_{\bar 0} \cap H_{\bar 0}\bigr) \oplus \bigl(\bigoplus\limits_{\alpha \in \Lambda^I_{\bar 0}} {\frak L}_{\alpha, {\bar 0}} \bigr)\Bigr) \oplus \Bigl(\bigl(I_{\bar 1} \cap H_{\bar 1}\bigr) \oplus \bigl(\bigoplus\limits_{\alpha \in \Lambda^I_{\bar 1}} {\frak L}_{\alpha, {\bar 1}} \bigr)\Bigr)
\end{equation}
${\rm where} \hspace{0.1cm} \Lambda_{\bar i}^I := \{\alpha \in  \Lambda_{\bar i} : I_{\bar i} \cap {\frak L}_{\alpha, {\bar i}} \neq 0 \},$ for ${\bar i} \in {\hu Z}_2.$

\begin{theorem}\label{teo3}
Let ${\frak L}$ be a split regular Hom-Lie superalgebra of maximal length, root multiplicative and with $\mathcal{Z}({\frak L})=0$. It holds that ${\frak L}$ is simple if and only if it has all its nonzero roots connected and $H = \sum\limits_{\alpha \in \Lambda}[{\frak L}_{\alpha}, {\frak L}_{-\alpha}]$.
\end{theorem}

\begin{proof}
The necessary implication is Theorem \ref{teo1}-2. To prove the
converse, consider $I$ a nonzero ideal of ${\frak L}$. By Lemmas \ref{navi}, \ref{navi2} and Equation \eqref{max} we can write $$I = \Bigl(\bigl(I_{\bar 0} \cap H_{\bar 0}\bigr) \oplus \bigl(\bigoplus\limits_{\alpha \in \Lambda^I_{\bar 0}} {\frak L}_{\alpha, {\bar 0}} \bigr)\Bigr) \oplus \Bigl(\bigl(I_{\bar 1} \cap H_{\bar 1}\bigr) \oplus \bigl(\bigoplus\limits_{\alpha \in \Lambda^I_{\bar 1}} {\frak L}_{\alpha, {\bar 1}} \bigr)\Bigr)$$ with $\Lambda_{\bar i}^I \subset \Lambda_{\bar i},$ for ${\bar i} \in {\hu Z}_2 $ and some $\Lambda_{\bar i}^I \neq
\emptyset$ by Lemma \ref{lema4}. Hence, we can take $\alpha_0 \in \Lambda_{\bar i}^I$ being so
\begin{equation}\label{I}
0 \neq {\frak L}_{\alpha_0, {\bar i}} \subset I,
\end{equation}
thus
\begin{equation}\label{I2}
\{{\frak L}_{\alpha_0 \phi^{z},\bar{i}}: z \in {\hu Z}\} \subset I
\end{equation}
because $\phi(I)=I$ together with Lemma \ref{lema1} and Remark \ref{rmk1a} allow us to assert that
\begin{equation}\label{casifinal}
\hbox{if $\alpha \in \Lambda^I_{\bar i}$ then $\{\alpha \phi^{z}: z \in {\hu Z}\}  \subset \Lambda^I_{\bar i}$}.
\end{equation}
Now, let us take any $\beta \in \Lambda \setminus \{\alpha_0\phi^{z},-\alpha_0\phi^{z} : z \in {\hu Z}\}$. Since $\alpha_0$ and $\beta$ are connected we have a connection $\{\alpha_1,\alpha_2,\dots,\alpha_k\} \subset \Lambda, k \geq 2$, from $\alpha_0$ to $\beta$ satisfying:
\begin{enumerate}
\item[{\rm 1.}] $\alpha_1 = \alpha_0\phi^{-n}$ for some $n \in {\mathbb N}_0$.

\item[{\rm 2.}] $\alpha_1\phi^{-1} + \alpha_2\phi^{-1} \in \Lambda$\\
$\alpha_1\phi^{-2} + \alpha_2\phi^{-2} + \alpha_3\phi^{-1} \in \Lambda$\\
$\hspace*{2cm} \vdots$\\
$\alpha_1\phi^{-i} + \alpha_2\phi^{-i} + \alpha_3\phi^{-i+1}+
\cdots + \alpha_{i+1}\phi^{-1} \in \Lambda$\\
$\hspace*{2cm} \vdots$\\
$\alpha_1\phi^{-k+2} + \alpha_2\phi^{-k+2} + \alpha_3\phi^{-k+3} + \cdots + \alpha_i\phi^{-k+i} + \cdots + \alpha_{k-1}
\phi^{-1} \in \Lambda$

\item[{\rm 3.}] $\alpha_1\phi^{-k+1} + \alpha_2\phi^{-k+1} + \alpha_3\phi^{-k+2} + \cdots + \alpha_i\phi^{-k+i-1} + \cdots + \alpha_k\phi^{-1} = \epsilon \beta\phi^{-m}$ for some $ m
\in {\mathbb N}_0$ and $\epsilon \in \{1,-1\}.$
\end{enumerate}

\noindent Consider $\alpha_1 = \alpha_0 \in \Lambda_{\bar i}$. Since $\alpha_2 \in \Lambda$ it follows ${\frak L}_{\alpha_2,\bar{j}} \neq 0,$ for some ${\bar j} \in {\hu Z}_2$, and so $\alpha_2 \in \Lambda_{{\bar j}}$. We have $\alpha_1 \in \Lambda_{{\bar i}}$ and $\alpha_2 \in \Lambda_{{\bar j}}$ such that $\alpha_1\phi^{-1} + \alpha_2\phi^{-1} \in \Lambda_{{\bar i} + \bar{j}}$ by Remark \ref{rmk1a}. The root-multiplicativity and maximal length of ${\frak L}$ allow us to assert $0 \neq [{\frak L}_{\alpha_1,\bar{i}} , {\frak L}_{\alpha_2,\bar{j}}] = {\frak L}_{\alpha_1 \phi^{-1}+ \alpha_2 \phi^{-1},\bar{i}+\bar{j}}.$ As consequence of Equation \eqref{I} we get $$0 \neq {\frak L}_{\alpha_1 \phi^{-1}+\alpha_2 \phi^{-1},\bar{i}+\bar{j}} \subset I.$$
A similar argument applied to $\alpha_1\phi^{-1}+\alpha_2
\phi^{-1}, \alpha_3$ and $$(\alpha_1\phi^{-1} + \alpha_2\phi^{-1})\phi^{-1} + \alpha_3\phi^{-1} = \alpha_1\phi^{-2} + \alpha_2\phi^{-2} + \alpha_3\phi^{-1}$$ gives us $0 \neq {\frak L}_{\alpha_1 \phi^{-2}+\alpha_2 \phi^{-2}+ \alpha_3 \phi^{-1},\bar{l}} \subset I,$ for certain $\bar{l} \in \mathbb{Z}_2.$ We can follow this process with the connection $\{\alpha_1,\alpha_2,\dots,\alpha_k\}$ to get $$0 \neq {\frak L}_{\alpha_1\phi^{-k+1}+ \alpha_2 \phi^{-k+1} + \alpha_3\phi^{-k+2} + \cdots + \alpha_i\phi^{-k+i-1} + \cdots+\alpha_k\phi^{-1},\bar{h}} \subset I,$$ for some $\bar{h} \in \mathbb{Z}_2,$ and then
\begin{equation}\label{either2}
\hbox{$0 \neq {\frak L}_{\epsilon\beta\phi^{-m}, \bar{h}} \subset I,$ for certain $\epsilon \in \{1,-1\}$}
\end{equation}
for any $\beta \in \Lambda \setminus \{\alpha_0\phi^{z},-\alpha_0\phi^{z} : z \in {\hu Z}\}$ and some $m \in \mathbb{N}_0.$ In general, from Equations \eqref{I2} and \eqref{either2} we get for any $\alpha \in \Lambda$
\begin{equation*}
\hbox{either $\{{\frak L}_{\alpha \phi^z,\bar{i}} : z \in {\hu Z}\} \subset I$ or $ \{ {\frak L}_{ -\alpha \phi^z,\bar{i}}: z \in {\hu Z}\} \subset I,$ for a fixed $\bar{i} \in \mathbb{Z}_2$.}
\end{equation*}
which can be reformulated by given any $\alpha \in \Lambda$:
\begin{equation}\label{either3}
\hbox{either $\{\alpha \phi^z : z \in \mathbb{Z}\}$ or $\{-\alpha \phi^z : z \in \mathbb{Z}\}$ is contained in $\Lambda^I_{\bar i},$ for some $\bar{i} \in \mathbb{Z}_2$.}
\end{equation}

From here, if we have $$\Lambda_{\bar 0} \cap \Lambda_{\bar 1} = \emptyset,$$ Equation \eqref{either3} shows that for any $\alpha \in \Lambda$ we have
\begin{equation}\label{equa1}
{\frak L}_{\epsilon \alpha} \subset I
\end{equation}
for some $\epsilon \in \{1,-1\}$. Since the expression of $H$ we get
\begin{equation}\label{equa2}
H \subset I.
\end{equation}
Now, for $-\epsilon\alpha \in \Lambda,$ the facts $-\epsilon\alpha \neq 0, H \subset I$ and the maximal length of ${\frak L}$ show
\begin{equation}\label{equa3}
[H_{\bar 0},{\frak L}_{-\epsilon\alpha}] = {\frak L}_{-\epsilon \alpha} \subset I.
\end{equation}
From Equations \eqref{equa1}-\eqref{equa3} we conclude $I={\frak L}$.

Therefore, it just remains to study the case in which there exists $\alpha \in \Lambda$ such that $$\alpha \in \Lambda_{\bar 0} \cap \Lambda_{\bar 1}.$$ We begin with the following observation.
\begin{equation}\label{obser}
\hbox{If there exists some $\alpha \in \Lambda_{\bar 0} \cap
\Lambda_{\bar 1}$ satisfying ${\frak L}_{\alpha, {\bar 0}} \oplus {\frak L}_{\alpha, {\bar 1}} \subset I$ then $I = {\frak L}$.}
\end{equation}
Indeed, by arguing as above with the fact that ${\frak L}$ has all its nonzero roots connected from ${\frak L}_{\alpha, {\bar 0}}$,  we get that for any $\beta \in \Lambda \setminus \{\alpha,-\alpha\}, 0 \neq {\frak L}_{\epsilon\beta,{\bar h}} \subset I$ for some $\epsilon \in \{1,-1\}$ and some ${\bar h} \in {\hu Z}_2$.
If we argue similarly from ${\frak L}_{\alpha, {\bar 1}}$ we obtain $0 \neq {\frak L}_{\epsilon \beta,{\bar h} + {\bar 1}} \subset I$. So ${\frak L}_{\epsilon \beta} \subset I.$
From here, as $H = \sum\limits_{\alpha \in \Lambda}[{\frak L}_{\alpha}, {\frak L}_{-\alpha}]$ we have
\begin{equation}\label{equa4}
H \subset I.
\end{equation}
Given now any $\gamma \in \Lambda$, the facts $\gamma \neq 0$,
$H \subset I$ and the maximal length of ${\frak L}$ show
$$[H_{\bar 0},{\frak L}_{\gamma,\bar{i}}] = {\frak L}_{\gamma,\bar{i}} \subset I$$ for any $\bar{i} \in \mathbb{Z}_2.$ Therefore
\begin{equation}\label{equa5}
[H_{\bar 0},{\frak L}_{\gamma}] = {\frak L}_{\gamma} \subset I.
\end{equation}
From Equations \eqref{equa4} and \eqref{equa5} we conclude $I = {\frak L}$.

Let us return to our $\alpha \in \Lambda$ such that $\alpha \in
\Lambda_{\bar 0} \cap \Lambda_{\bar 1}$. Taking into account
Equation \eqref{either3} $\epsilon \alpha\phi^{-m} \in \Lambda,$ for certain $\epsilon \in \{1,-1\}$ and some $m \in \mathbb{N}_0,$ then $[H_{\bar 0}, {\frak L}_{\epsilon \alpha\phi^{-m}, {\bar h}}]\neq 0$ for fixed $\bar{h} \in \mathbb{Z}_2$. Since $H_{\bar 0} = \sum_{\gamma\in \Lambda}([{\frak L}_{\gamma, {\bar 0}}, {\frak L}_{-\gamma, {\bar 0}}] + [{\frak L}_{\gamma, {\bar 1}}, {\frak L}_{-\gamma, {\bar 1}}])$, there exists $\gamma \in \Lambda$ such that either \newline $[[{\frak L}_{\gamma, {\bar 0}}, {\frak L}_{-\gamma, {\bar 0}}], {\frak L}_{\epsilon \alpha\phi^{-m}, {\bar h}}] \neq 0$ or $[[{\frak L}_{\gamma, {\bar 1}}, {\frak L}_{-\gamma, {\bar 1}}], {\frak L}_{\epsilon \alpha\phi^{-m}, {\bar h}}] \neq 0.$ That is, by Lemma \ref{lema1} and Remark \ref{remarkr} is $\phi({\frak L}_{\epsilon \beta\phi^{-m+1},{\bar h}}) = {\frak L}_{\epsilon \alpha\phi^{-m},{\bar h}},$ and so
\begin{equation}\label{paracero}
[[{\frak L}_{\gamma, {\bar n}}, {\frak L}_{-\gamma, {\bar n}}], \phi({\frak L}_{\epsilon \alpha\phi^{-m+1},{\bar h}})]\neq 0
\end{equation}
for some ${\bar n} \in {\hu Z}_2$ and fixed $\bar{h} \in {\hu Z}_2, m \in \mathbb{N}_0$. By super Hom-Jacobi identity $$\hbox{either $[{\frak L}_{\epsilon \alpha\phi^{-m+1},{\bar h}},{\frak L}_{\gamma, {\bar n}}] \neq 0$ or $[{\frak L}_{-\gamma, {\bar n}}, {\frak L}_{\epsilon \alpha\phi^{-m+1},{\bar h}}] \neq 0$}$$ and so ${\frak L}_{\gamma + \epsilon \alpha\phi^{-m+1}, {\bar n}+ {\bar h}} \neq 0$ or ${\frak L}_{-\gamma+ \epsilon \alpha\phi^{-m+1}, {\bar n}+ {\bar h}} \neq 0$. That is, by Equation \eqref{either2},
\begin{equation}\label{final}
0 \neq {\frak L}_{\kappa \gamma+ \epsilon \alpha\phi^{-m+1}, {\bar n}+ {\bar h}} \subset I
\end{equation}
for some $\kappa \in \{1,-1\}$.

Let us distinguish two possibilities, if ${\frak L}_{-\kappa \gamma, {\bar n} +{\bar 1}} \neq 0$, we get by the root-multiplicativity and maximal length of ${\frak L}$ that $$0 \neq [{\frak L}_{\kappa \gamma+ \epsilon \alpha\phi^{-m+1},{\bar n} + {\bar h}}, {\frak L}_{-\kappa \gamma, {\bar n}+{\bar 1}}] = {\frak L}_{\epsilon \alpha\phi^{-m+1}, {\bar h}+{\bar 1}} \subset I.$$ From here, Equations \eqref{either3} and \eqref{obser} give us $I = {\frak L}$.

If ${\frak L}_{-\kappa \gamma, {\bar n} + {\bar 1}}= 0$, as by symmetry of $\Lambda$ is $-\epsilon \alpha\phi^{-m} \in \Lambda,$ and by Lemma \ref{lema2} we also obtain $-\epsilon \alpha\phi^{-m+1} \in \Lambda,$ so ${\frak L}_{-\epsilon \alpha\phi^{-m+1}, {\bar m}} \neq 0$ for some ${\bar m} \in {\hu Z}_2$. By Equation \eqref{final} and the root-multiplicativity and maximal length of ${\frak L}$ we obtain
\begin{equation}\label{dieciseis}
0 \neq [{\frak L}_{\kappa \gamma+ \epsilon \alpha\phi^{-m+1}, {\bar n} +{\bar h}}, {\frak L}_{-\epsilon \alpha\phi^{-m+1}, {\bar m}}] = {\frak L}_{\kappa \gamma, {\bar n}+{\bar h}+{\bar m}} \subset I.
\end{equation}

If ${\bar h}+{\bar m}={\bar 0}$, taking into account that Equation \eqref{paracero} gives us $[[{\frak L}_{\gamma, {\bar n}}, {\frak L}_{-\gamma, {\bar n}}], {\frak L}_{\epsilon \alpha\phi^{-m},{\bar h}}]\neq 0,$ that is,
$$\alpha([{\frak L}_{\gamma, {\bar n}}, {\frak L}_{-\gamma, {\bar n}}])\neq 0,$$
Equation \eqref{dieciseis} and the fact $\alpha \in \Lambda_{\bar 0} \cap \Lambda_{\bar 1}$ let us assert $0 \neq [[{\frak L}_{\gamma, {\bar n}}, {\frak L}_{-\gamma, {\bar n}}], {\frak L}_{\alpha, {\bar p}}] = {\frak L}_{\alpha, {\bar p}}\subset I$ for any ${\bar p} \in {\hu Z}_2$. From here, Equation \eqref{obser} gives us $I = {\frak L}$.

If ${\bar h}+{\bar m}={\bar 1}$ we have by Equation \eqref{dieciseis} and the root-multiplicativity and
maximal length of ${\frak L}$ that $0 \neq [{\frak L}_{\kappa \gamma, {\bar n} + {\bar 1}}, {\frak L}_{\epsilon \alpha\phi^{-m+1}, {\bar h}}] = {\frak L}_{\kappa \gamma + \epsilon \alpha\phi^{-m+1},{\bar n} + {\bar h}+ {\bar 1}}\subset I$. From here, Equations \eqref{final} and \eqref{obser} give us $I = {\frak L}$. Consequently ${\frak L}$ is simple.
\end{proof}

\begin{theorem}
Let ${\frak L}$ be a split regular Hom-Lie superalgebra of maximal length, root multiplicative satisfying $\mathcal{Z}({\frak L})=0$ and $H = \sum_{\alpha \in \Lambda}[{\frak L}_{\alpha}, {\frak L}_{-\alpha}].$ Then ${\frak L}$ is the direct sum of the family of its minimal ideals, each one being a simple split regular Hom-Lie superalgebra having all its nonzero roots connected.
\end{theorem}

\begin{proof}
By Corollary \ref{coro1}, ${\frak L} = \bigoplus_{[\alpha] \in \Lambda/\sim} {\frak L}_{[\alpha]}$ is the direct sum of the  ideals $${\frak L}_{[\alpha]} = H_{[\alpha]} \oplus V_{[\alpha]} = \Bigl(\sum\limits_{\beta \in [\alpha]} [{\frak L}_{\beta},{\frak L}_{-\beta}]\Bigr) \oplus \Bigl(\bigoplus\limits_{\beta \in [\alpha]} {\frak L}_{\beta}\Bigr)$$ having any ${\frak L}_{[\alpha]}$ its root system $[\alpha]$ with all of its roots connected. Taking into account the facts $[\alpha] = -[\alpha]$ and $[{\frak L}_{[\alpha]}, {\frak L}_{[\alpha]}] \subset {\frak L}_{[\alpha]}$ (see Proposition \ref{pro2}-1) we easily deduce that $[\alpha]$ has all of its roots connected through roots in $[\alpha]$. We also have that any of the ${\frak L}_{[\alpha]}$ is root-multiplicative as consequence of the root-multiplicativity of ${\frak L}$ and trivially is regular. Finally, is clear that ${\frak L}_{[\alpha]}$ is of maximal length and ${\mathcal Z}_{{\frak L}_{[\alpha]}}({\frak L}_{[\alpha]})=0$ (where ${\mathcal
Z}_{{\frak L}_{[\alpha]}}({\frak L}_{[\alpha]})$ denotes de center of ${\frak L}_{[\alpha]}$ in ${\frak L}_{[\alpha]}$) as consequence of $[{\frak L}_{[\alpha]},{\frak L}_{[\gamma]}]=0$ if $[\alpha] \neq [\gamma]$ (Theorem \ref{teo2}) and ${\mathcal Z}({\frak L}) = 0$. We can apply Theorem \ref{teo3} to any ${\frak L}_{[\alpha]}$ so as to conclude ${\frak L}_{[\alpha]}$ is simple. It is clear that the decomposition ${\frak L} =\bigoplus_{[\alpha] \in \Lambda/\sim} {\frak L}_{[\alpha]}$ satisfies the assertions of the theorem.
\end{proof}

\begin{remark}\label{Ree}
We finally note that the results in the present paper also hold
for non regular Hom-Lie superalgebras ${\frak L}$ if we introduce the class of split non regular Hom-Lie superalgebras by asking for $H$ not only the condition of being a maximal abelian subalgebra of ${\frak L}$ but also the one of satisfying that $\phi|_H$ is a linear bijection from $H$ onto $H.$
\end{remark}

\medskip

\begin{example}
Let us consider the real  ${\mathbb Z}_2$-graded vector space
${\frak L}={\frak L}_{\bar 0} \oplus {\frak L}_{\bar 1}$,  with
basis $$\{e_1,e_2\} \cup \{h_n,x_n,y_n: n \in {\mathbb N}, n \geq 2\} $$ of ${\frak L}_{\bar 0}$ and  $$\{e_3\}\cup \{f_n,g_n: n \in {\mathbb N}, n \geq 2\} $$
of ${\frak L}_{\bar 1}$ given in Example \ref{example1}. Here  the nonzero products on these elements
were induced  by the relations:
$$[e_2,e_1]=e_1, \hspace{0.3cm} [h_n,x_n]=2n^2x_n, \hspace{0.3cm}
[h_n,y_n]=- \frac{2}{n^2}y_n,$$ $$[x_n,y_n]=h_n,\hspace{0.3cm}
[y_n,g_n]=\frac{1}{n}f_n, \hspace{0.3cm}[x_n,f_n]={n}g_n,$$
$$[h_n,f_n]=-\frac{1}{n}f_n,\hspace{0.3cm} [h_n,g_n]={n}g_n,
 \hspace{0.3cm}[g_n,f_n]=h_n,\hspace{0.3cm}  [g_n,g_n]=-2n^2x_n,$$
 $$[f_n,f_n]=\frac{2}{n^2}y_n.$$

  If we now consider  the  superspace homomorphism $$\phi^{\prime}: {\frak L} \to {\frak L}$$ defined by

  $$\hbox{$\phi^{\prime}(e_1)=e_1$,  $\phi^{\prime}(e_2)=e_2$, $\phi^{\prime}(e_3)=0$ and }$$
  $$\hbox{$\phi^{\prime}(h_n)=h_n$, $\phi^{\prime}(x_n)=n^2 x_n$, $\phi^{\prime}(y_n)=\frac{1}{n^2} y_n$, $\phi^{\prime}(f_n)=\frac{1}{n} f_n$, $\phi^{\prime}(g_n)=n g_n$ }$$
   for any $n \geq 2$,
  we have that
  ${\frak L}={\frak L}_{\bar 0} \oplus {\frak L}_{\bar
 1}$ also becomes a split Hom-Lie superalgebra, with a root spaces decomposition as in Equation (\ref{sep1}),  which is not regular since $\phi^{\prime}$ is not an automorphism of $({\frak L}, [\cdot, \cdot])$. However, $\phi^{\prime}|_H$ is a linear bijection from $H$ onto $H.$ From here,   Remark \ref{Ree} allows us to assert that ${\frak L}$ admits an analogous decomposition as sum of ideals as the one given   in Equation (\ref{sep2}).

\end{example}

\bigskip

{\bf Acknowledgment.} We would like to thank  the referee  for the detailed reading of this work and for the suggestions which have improved the final version of the same.

\end{document}